\documentclass[a4paper,12pt,reqno]{amsart}
\usepackage{amsmath, amsthm, amssymb, amstext}

\usepackage[left=2.3cm,right=2.3cm,top=2.5cm,bottom=2.5cm]{geometry}

\usepackage{hyperref,xcolor}
\usepackage{tikz}
\usepackage[hyperpageref]{backref}
\hypersetup{pdfborder={0 0 0},colorlinks}

\usepackage{enumitem}
\allowdisplaybreaks
\newtheorem{theorem}{Theorem}
\newtheorem{remark}[theorem]{Remark}
\newtheorem{lemma}[theorem]{Lemma}

\newtheorem{corollary}[theorem]{Corollary}

\numberwithin{theorem}{section}
\numberwithin{equation}{section}

\title[Boundary geometry in nonlocal critical problems]
{The effect of boundary geometry in nonlocal critical problems with Hardy-Littlewood-Sobolev exponent}

\author[H. Chtioui]{Hichem Chtioui}
\address[H. Chtioui]{Department of Mathematics, Faculty of Sciences of Sfax, Sfax University, Tunisia}
\email{hichem.chtioui@fss.rnu.tn}

\author[T. Mukherjee]{Tuhina Mukherjee}
\address[T. Mukherjee]{Department of Mathematics, Indian Institute of Technology Jodhpur, Rajasthan 342030, India}
\email{tuhina@iitj.ac.in}

\author[L. Sharma]{Lovelesh Sharma}
\address[L. Sharma]{Department of Mathematics, Indian Institute of Technology Jodhpur, Rajasthan 342030, India}
\email{sharma.94@iitj.ac.in}


\begin{document}

\begin{abstract}
In this paper we consider a mixed Dirichlet-Neumann boundary value problem involving Choquard nonlinearity with upper critical exponent in the sense of Hardy-Littlewood-Sobolev inequality. We investigate the effect of the geometry of the boundary part where the Neumann condition is prescribed on the existence problem of ground state solutions.
\end{abstract}

\maketitle

\section{Introduction}
Let \(\Omega\) be a bounded domain of \(\mathbb{R}^n\), \(n \geq 3\), whose boundary \(\partial\Omega\) is Lipschitz continuous and is given by the union of two disjoint \((n-1)\)-dimensional manifolds \(\Gamma_0\) and \(\Gamma_1\). Suppose that the Hausdorff measure \(H_{n-1}(\Gamma_0)\) of the boundary part \(\Gamma_0\) is positive. The present article focuses on the existence problem of nontrivial solutions of the following Choquard equation with mixed Dirichlet-Neumann boundary condition
\begin{equation}
\begin{cases}
-\Delta u + V(x)u = (I_\mu * |u|^{2^*_\mu}) |u|^{2^*_\mu-2} u & \text{in } \Omega, \\
~ u = 0 & \text{on } \Gamma_0, \\
\frac{\partial u}{\partial \nu} = 0 & \text{on } \Gamma_1,
\end{cases}\tag{P}
\label{eq:main-problem}
\end{equation}
where \(\nu\) is the unit outer normal to \(\Gamma_1\), \(\mu \in (0, n)\), if \(n =\) \(3\) or \(4\), \(\mu \in (0, 4]\), if \(n \geq 5\), \(2^*_\mu = \frac{2n - \mu}{n - 2}\) is the upper critical exponent in the sense of Hardy-Littlewood-Sobolev, \(I_\mu : \mathbb{R}^n \to \mathbb{R}\) is the Riesz potential defined by
\[
I_\mu(x) = \frac{\Gamma(\frac{\mu}{2})}{\Gamma(\frac{n-\mu}{2}) \pi^{\frac{n}{2}} 2^{n-\mu}|x|^{\mu}} , \quad x \in \mathbb{R}^n \setminus \{0\},
\]
and \(V : \overline{\Omega} \to \mathbb{R} \in C(\overline{\Omega})\) is a given external potential such that \(-\Delta + V\) is coercive with respect to the usual norm of \(H^1_0(\Omega)\).

Choquard equations have attracted the interest of a lot of scientists due to their numerous applications in various fields of mathematical physics. In 1954,  Pekar \cite{pekar1954untersuchungen} introduced the Choquard-Pekar equation
\begin{equation}\label{eq:pekar}
-\Delta u + u = (I_1 * |u|^2) u \quad \text{in } \mathbb{R}^3,
\end{equation}
to model the quantum mechanics of a polaron at rest. In 1976, Ph. Choquard used equation \eqref{eq:pekar} in the modelling of an electron trapped in its own hole as a certain approximation to Hartree-Fock theory of one component plasma \cite{lieb1977existence}.

A further application has been investigated by  Penrose \cite{penrose1996gravity}, who used equation \eqref{eq:pekar} to model self-gravitational collapse of a quantum mechanical wave function.
A considerable family of equations which extends \eqref{eq:pekar} on \(\mathbb{R}^n\), \(n \geq 3\), is given by
\begin{equation}\label{eq:general-choquard}
-\Delta u + u = (I_\mu * |u|^p) |u|^{p-2} u \quad \text{in } \mathbb{R}^n,
\end{equation}
where \(p > 1\).
The solutions of \eqref{eq:general-choquard} are the critical points of the energy functional
\[
I(u) = \frac{1}{2} \int_{\mathbb{R}^n} (|\nabla u|^2 + u^2) - \frac{1}{2p} \int_{\mathbb{R}^n} (I_\mu * |u|^p) |u|^p.
\]
If we use as a domain of \(I\) the Sobolev space \(H^1(\mathbb{R}^n)\), the second term of \(I(u)\) is defined if and only if
\[
2_{\mu} = \frac{2n - \mu}{n} \leq p \leq 2^*_\mu = \frac{2n - \mu}{n - 2}.
\]
This is a consequence of the celebrated Hardy-Littlewood-Sobolev inequality. The numbers \(2_{\mu}\) and \(2^*_\mu\) are called respectively the lower critical exponent and the upper critical exponent.

In the presence of an external potential \(V: \mathbb{R}^n \to \mathbb{R}\), the Choquard equation
\begin{equation}\label{eq:potential-choquard}
-\Delta u + V(x) u = (I_\mu * |u|^p) |u|^{p-2} u, \quad \text{in } \mathbb{R}^n,
\end{equation}
\(p \in [2_\mu, 2^*_\mu]\), has been the subject of multiple papers. When \(V\) is a perturbation of a positive constant and \(p = 2\), equation \eqref{eq:potential-choquard} has been studied by  Lions \cite{lions1980choquard}. Using his concentration-compactness principle, Lions proved some criteria for existence results. For a positive constant potential \(V\), \(p = 2\), \(\mu = 1\) and \(n = 3\),  Lieb \cite{lieb1977existence} proved the existence of a unique minimizing solution (up to translations), the so-called ground state solution.

Here and in the sequel, we define a solution \(u\) to be a ground state of problem \eqref{eq:main-problem} and related ones, whenever it is a nontrivial solution that minimizes the associated energy functional over all nontrivial solutions.

For \(V(x) = 1\), Moroz et al. \cite{moroz2013groundstates} proved a necessary and sufficient condition for obtaining ground state solutions to problem \eqref{eq:potential-choquard}; that is
\(
2_\mu < p < 2^*_\mu.
\)
The case of the lower critical exponent \(p = 2_\mu\) and a non-constant potential \(V\) has been considered by Moroz et al. \cite{moroz2015groundstates} and Cassani et al. \cite{cassani2020groundstates}, where some conditions on the potential \(V(x)\) are imposed to obtain ground state solutions.

On bounded domains \(\Omega\) of \(\mathbb{R}^n\), \(n \geq 3\), most studies on Choquard equations have been established under homogeneous Dirichlet boundary conditions. In \cite{goel2020coron},  R\u adulescu et al. considered the case where \(\Omega\) is a bounded annular- type domain. By a variational approach, they proved the existence of positive solutions when the inner hole of \(\Omega\) is sufficiently small. In \cite{alghamdi2024nonlocal}, the first author and M. Alghamdi extended the result of \cite{goel2020coron} via a topological method to any bounded domain \(\Omega\), provided that \(H_k(\Omega, \mathbb{Z}_2) \neq 0\) for some positive integer \(k\). In \cite{gao2016brezis},  Gao and  Yang studied a critical Brezis-Nirenberg type problem with Choquard nonlinearity on \(\Omega\). In \cite{gao2017nonlocal, goel2019critical, squassina2023local}, it is shown that the method of Brezis--Nirenberg \cite{brezis1983positive} developed for local elliptic equations is successfully adapted to the study of homogeneous Dirichlet--Choquard equations. For recent papers on the Choquard problem, we refer the reader to \cite{bernardini2025boundary, chen2021positive, chen2025static, squassina2023local, liang2020multiple, liang2024critical, giacomoni2023critical}.

The purpose of the present paper is to study a class of Choquard equations on bounded domains of \(\mathbb{R}^n\), \(n \geq 3\), under mixed Dirichlet-Neumann boundary conditions. Motivated by the idea developed in the work of Adimurthi and Mancini \cite{adimurthi1991neumann} on critical local equations, we investigate the effect of the geometry of the boundary part \(\Gamma_1\) on the existence of ground state solutions for problem \eqref{eq:main-problem}. We point out here that in a recent paper of  Giacomoni et al. \cite{giacomoni2023critical}, ground state solutions for problem \eqref{eq:main-problem} have been proved under suitable conditions on the external potential \(V(x)\) and a flatness condition on \(\Gamma_1\).

\vspace{0.2cm}

Now we introduce the following local geometrical condition.

Let \(a \in \Omega\). Without loss of generality, we may assume that \(a = 0_{\mathbb{R}^n}\). Let \(\varphi(x')\), \(x' \in \mathbb{R}^{n-1}\), be a local parametrization of \(\Gamma_1\) near \(a\). Therefore, for \(\rho > 0\) small enough, we have
\[
B(a,\rho) \cap \Gamma_1 = \bigl\{(x', x_n) \in B(a,\rho) \; \text{such that} \; x_n = \varphi(x')\bigr\},
\]
where \(B(a,\rho)\) is the ball of \(\mathbb{R}^n\) centred at \(a\) and of radius \(\rho\). In a generic case, we may assume that in \(B(a,\rho)\), the domain \(\Omega\) lies on one side of the tangent space of \(\Gamma_1\) at \(a\). Up to a change of coordinates, we can assume that
\[
B(a,\rho) \cap \Omega = \bigl\{(x', x_n) \in B(a,\rho) \; \text{such that} \; x_n > \varphi(x')\bigr\}.
\]

\(\mathbf{(H_1)}\): Assume that there exists a real number \(\beta \in (1, n-1]\) such that
\begin{equation}\label{1.1}
\varphi(x') = \sum_{k=1}^{n-1} \gamma_k |x'_k|^\beta + o(|x'|^\beta),
\end{equation}
with \(\sum_{k=1}^{n-1} \gamma_k > 0\).

\begin{remark}
    The restriction of the flatness order \(\beta\) in \eqref{1.1} to be less than or equal to \(n-1\) is due to a technical reason. Indeed, in the expansions of Lemma \ref{lem3.1}, the leading term becomes \(\frac{\sum_{k=1}^{n-1} \gamma_k}{\lambda^{\beta-1}}\) if \(\beta < n-1\) and \(\frac{(\sum_{k=1}^{n-1} \gamma_k) \log \lambda}{\lambda^{\beta-1}}\) if \(\beta = n-1\). For \(\beta > n-1\), the first term of the expansion becomes \(\frac{\sum_{k=1}^{n-1} \gamma_k}{\lambda^{n-2}}\), which is no longer leading over the error terms.
\end{remark}

\begin{remark}
    If \(\varphi\) is of class \(C^\infty\), then \(\beta\) is a positive integer. For \(\beta = 2\), the quantity \(\sum_{k=1}^{n-1} \gamma_k\) corresponds (up to a positive multiplicative constant) to the mean curvature of \(\Gamma_1\) at \(a\). Therefore, condition \(\mathbf{(H_1)}\) covers the hypothesis of [\cite{adimurthi1991neumann}, Theorem 1.2] for the related local equations.
\end{remark}

We shall prove the following existence results. The first theorem does not impose any further condition on \(V(x)\).

\begin{theorem}\label{th1.3}
    Assume condition \(\mathbf{(H_1)}\) with \(\beta < 3\). Then problem \eqref{eq:main-problem} admits a ground state solution.
\end{theorem}

The second theorem is valid for any flatness order \(\beta\), \(1 < \beta \leq n-1\).

\begin{theorem}\label{th1.4}
    Under condition \(\mathbf{(H_1)}\), if
    \[
    V(x) \leq 0 \quad \forall\, x \in B(a,\rho) \cap \Omega,
    \]
    then problem \eqref{eq:main-problem} possesses a ground state solution.
\end{theorem}

The proofs of Theorems \ref{th1.3} and \ref{th1.4} will be given in Section 3. The method follows the idea of Adimurthi and Mancini \cite{adimurthi1991neumann}, which is based on the minimization technique of  T. Aubin \cite{aubin1976equations}. It consists in proving, by means of a suitable test function, that the infimum of the associated energy functional is strictly below the first level at which the Palais-Smale condition fails. This guarantees the convergence of a minimizing sequence to a nontrivial critical point.

According to the preceding results, we present an example of domains for which problem \eqref{eq:main-problem} has ground state solutions.

\begin{corollary}\label{C01}
    Let \(\Omega_1 \subset \mathbb{R}^n\), \(n \geq 3\), be a bounded domain with smooth boundary \(\partial \Omega_1\), and let \(\Omega_0\) be an open subset of \(\Omega_1\) such that \(\overline{\Omega_0} \subset \Omega_1\). Set
    \[
    \Omega = \Omega_1 \setminus \overline{\Omega_0}, \qquad \Gamma_0 = \partial \Omega_0,\; \Gamma_1 = \partial \Omega_1.
    \]
    Then problem \eqref{eq:main-problem} admits a positive ground state solution.
\end{corollary}

\begin{figure}[h]
    \centering
    \begin{tikzpicture}[scale=0.8]
    % Shaded region Omega = Omega_1 \setminus \overline{Omega_0}
    \fill[gray!30, even odd rule]
        (0,0) circle (3)
        (0,0) circle (1.2);

    % Outer boundary Gamma_1
    \draw[thick, blue] (0,0) circle (3);
    \node[blue] at (2,2) {$\Omega_1$};

    % Inner boundary Gamma_0
    \draw[thick, red] (0,0) circle (1.2);
    \node[red] at (0.6,0.6) {$\Omega_0$};

    % Labels
    \node at (0,1.5) {$\Omega = \Omega_1 \setminus \overline{\Omega_0}$};
    \node[blue] at (0,-2.8) {$\Gamma_1 = \partial\Omega_1$};
    \node[red] at (0,-1.3) {$\Gamma_0 = \partial\Omega_0$};
\end{tikzpicture}
    \caption{Domain configuration for Corollary \ref{C01}}
\end{figure}
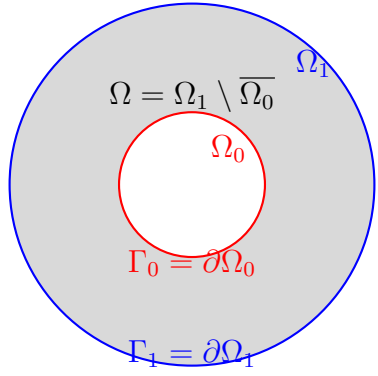
We point out that the choice of the boundary part \(\Gamma_1\) in Corollary \ref{C01} as the boundary of \(\Omega_1\) guarantees the existence of a point \(a\in \Gamma_1\) where the mean curvature of \(\Gamma_1\) is positive. Hence condition \(\mathbf{(H_1)}\) holds around \(a\) with \(\beta = 2\). Thus, Corollary \ref{C01} follows from Theorem \ref{th1.3}.

However, if \(\Gamma_1\) is taken as the boundary \(\partial \Omega_0\), condition \(\mathbf{(H_1)}\) fails in general (for instance, in the case of an annular domain). This leads to the following natural question:

\noindent\textbf{Question:} Let \(\Omega\) be a connected bounded domain as in Corollary \ref{C01}, but with \(\Gamma_0 = \partial \Omega_1\) and \(\Gamma_1 = \partial \Omega_0\). Does problem \eqref{eq:main-problem} possess a nontrivial solution?
\section{Variational analysis}
First, we recall the Hardy-Littlewood-Sobolev inequality.

\begin{lemma}[\cite{lieb2001analysis}] \label{lem:HLS}
    Let \(t_1, t_2 > 1\) satisfy \(\frac{1}{t_1} + \frac{1}{t_2} = \frac{2n - \mu}{n}\). Then,
    \begin{equation}\label{eq:HLS-inequality}
        \int_{\mathbb{R}^n} \int_{\mathbb{R}^n} \frac{f(y)\, g(x)}{|y-x|^\mu} \, dx\, dy
        \leq C \|f\|_{L^{t_1}(\mathbb{R}^n)} \|g\|_{L^{t_2}(\mathbb{R}^n)},
    \end{equation}
    where \(C\) is a positive constant independent of \(f\) and \(g\). Moreover, if
    \(t_1 = t_2 = \frac{2n}{2n-\mu}\), then
    \[
        C = \pi^{\frac{\mu}{2}} \frac{\Gamma\!\bigl(\frac{n-\mu}{2}\bigr)}{\Gamma\!\bigl(n-\frac{\mu}{2}\bigr)}
        \cdot \frac{\Gamma\!\bigl(\frac{n}{2}\bigr)}{\Gamma(n)}.
    \]
    In this case, equality holds in \eqref{eq:HLS-inequality} if and only if
    \[
        f(x) = g(x) = A \bigl( \lambda^2 + |x-a|^2 \bigr)^{-\frac{2n-\mu}{2}}, \quad x \in \mathbb{R}^n,
    \]
    up to a multiplicative constant, where \(A \in \mathbb{C}\), \(\lambda > 0\) and \(a \in \mathbb{R}^n\).
\end{lemma}
Note that under the above inequality, the integral
\begin{equation}\label{eq2.2}
\int_{\mathbb{R}^n} \int_{\mathbb{R}^n} \frac{|u(y)|^q |u(x)|^q}{|x-y|^\mu} \, dx\, dy
\end{equation}
is well defined provided \(|u|^q \in L^t(\mathbb{R}^n)\) with \(t = \frac{2n}{2n-\mu}\). It follows from the Sobolev inequality that \eqref{eq2.2} is well defined on \(H^1(\mathbb{R}^n)\) only for
\[
\frac{2n - \mu}{n} \leq q \leq \frac{2n - \mu}{n-2}.
\]
In the case of the upper critical Hardy-Littlewood-Sobolev exponent \(2^*_\mu = \frac{2n - \mu}{n-2}\), it is proved in [\cite{gao2016brezis}, Lemma 2.3] that for any bounded domain \(\Omega\) of \(\mathbb{R}^n\), \(n\geq 3\),
\[
\|u\|_{\mathrm{HLS}} := \left( \int_{\Omega} \int_{\Omega} \frac{|u(y)|^{2^*_\mu} |u(x)|^{2^*_\mu}}{|x-y|^\mu} \, dx\, dy \right)^{\frac{1}{2\cdot 2^*_\mu}}
\]
is a norm on \(X_{\mathrm{HLS}} := \{ u : \Omega \to \mathbb{R} \mid \|u\|_{\mathrm{HLS}} < \infty \}\). Moreover, the best constant
\begin{equation}\label{eq2.3}
S_{\mathrm{HL}} := \inf_{u \in H^1_0(\Omega) \setminus \{0\}} \frac{\int_{\Omega} |\nabla u|^2 \, dx}{\|u\|_{\mathrm{HLS}}^2}
\end{equation}
is independent of \(\Omega\) and is never achieved except when \(\Omega = \mathbb{R}^n\). In the case \(\Omega = \mathbb{R}^n\), the unique minimizers for \(S_{\mathrm{HL}}\) are the functions of the form
\[
u(x) = c \, \delta_{(a,\lambda)}(x), \quad x \in \mathbb{R}^n,
\]
where \(c\) is a positive constant and, for \(a \in \mathbb{R}^n\) and \(\lambda > 0\),
\begin{equation}\label{eq2.4}
\delta_{(a,\lambda)}(x) = (n(n-2))^{\frac{n-2}{4}} \frac{\lambda^{\frac{n-2}{2}}}{\bigl(1 + \lambda^2 |x-a|^2 \bigr)^{\frac{n-2}{2}}}.
\end{equation}
See [\cite{gao2016brezis}, Lemma 1.2]. Let
\[
S = \inf_{u \in H^1_0(\Omega) \setminus \{0\}} \frac{\int_{\Omega} |\nabla u|^2 \, dx}{\|u\|^2_{L^{\frac{2n}{n-2}}(\Omega)}}
\]
be the best Sobolev constant, and let
\[
C_{n,\mu} = \frac{1}{2^{n-\mu} \pi^{\frac{n-\mu}{2}}}
\frac{\Gamma\!\bigl(\frac{\mu}{2}\bigr)}{\Gamma\!\bigl(n-\frac{\mu}{2}\bigr)}
\left( \frac{\Gamma(n)}{\Gamma\!\bigl(\frac{n}{2}\bigr)} \right)^{\frac{n-\mu}{n}}.
\]
It is proved in [\cite{du2018uniqueness}, Lemma 1.2] (see also \cite{gao2016brezis}) that
\begin{equation}\label{eq2.5}
U_{(a,\lambda)}(x) = \left( S^{\frac{n-\mu}{2}} C_{n,\mu} \right)^{\frac{2-n}{2(n-\mu+2)}} \delta_{(a,\lambda)}(x)
= \widetilde{C}_{n,\mu} \, \delta_{(a,\lambda)}(x),
\end{equation}
with \(a \in \mathbb{R}^n\) and \(\lambda > 0\), is the unique family of positive solutions of
\begin{equation}\label{eq2.6}
-\Delta u = \bigl( I_\mu * U^{2^*_\mu} \bigr) U^{2^*_\mu-1} \quad \text{in } \mathbb{R}^n.
\end{equation}
We now define
\[
V(\Omega) = \bigl\{ u \in H^1(\Omega) \mid u = 0 \text{ on } \Gamma_0 \bigr\}.
\]
Under our assumptions,
\[
\|u\|_{V(\Omega)} := \left( \int_{\Omega} |\nabla u|^2 \, dx + \int_{\Omega} V(x) u^2 \, dx \right)^{\frac{1}{2}}
\]
is a norm on \(V(\Omega)\) equivalent to \(|u|_{1,\Omega} := \bigl( \int_{\Omega} |\nabla u|^2 \, dx \bigr)^{\frac{1}{2}}\).

The energy functional associated with problem \eqref{eq:main-problem} is
\begin{equation}\label{eq2.7}
J(u) = \frac{\|u\|_{V(\Omega)}^2}{\|u\|_{\mathrm{HLS}}^2}, \qquad u \in V(\Omega) \setminus \{0\}.
\end{equation}
Up to multiplication by a positive constant, a critical point of \(J\) is a nontrivial solution of \eqref{eq:main-problem}, and a minimizer of \(J\) is a ground state solution of \eqref{eq:main-problem}. Setting
\begin{equation}\label{eq2.8}
S_{\mathrm{HL}}(\Gamma_0) := \inf_{u \in V(\Omega) \setminus \{0\}} J(u).
\end{equation}
The Hardy-Littlewood-Sobolev inequality together with the equivalence of \(\|\cdot\|_{V(\Omega)}\) and \(|\cdot|_{1,\Omega}\) in \(V(\Omega)\) implies that \(S_{\mathrm{HL}}(\Gamma_0)\) is positive. Moreover, if \(u_0\) is a minimizer of \(S_{\mathrm{HL}}(\Gamma_0)\), then \(|u_0|\) is also a minimizer, because \(|u_0| \in V(\Omega) \setminus \{0\}\) and \(J(|u_0|) = J(u_0)\). Consequently, both \(u_0\) and \(|u_0|\) are ground state solutions of problem \eqref{eq:main-problem} (up to a positive multiplicative constant).
\section{Proof of the existence result}
This section is devoted to the proof of Theorems \ref{th1.3} and \ref{th1.4}.
We prove that, under the assumptions of the theorems, the minimization constant
\(S_{HL}(\Gamma_0)\), defined in \eqref{eq2.8}, is achieved.
To this end, we adopt some ideas from \cite{adimurthi1991neumann}, which are based on
minimization arguments from \cite{aubin1976equations}. It consists to prove by using a suitable test function that \(S_{HL}(\Gamma_0)\) is strictly less than the first energy level at which the Palais--Smale condition fails. This ensures the convergence of minimizing
sequences in \(V(\Omega)\setminus \{0\}\).

Namely, it is proved in \cite[Lemma 3.2]{giacomoni2023critical} that
\(S_{HL}(\Gamma_0)\) is attained provided that
\begin{equation}\label{eq3.1}
S_{HL}(\Gamma_0)
< \left( \frac{1}{2} \right)^{\frac{2^*_\mu-2}{2^*_\mu}} S_{HL}.
\end{equation}
Let \(a \in \Gamma_1\) be a point satisfying the condition \(\mathbf{(H_1)}\).
Up to a translation and a rotation, we may assume that
\(a = 0_{\mathbb{R}^n}\). For \(\rho > 0\) sufficiently small, we define the following
cut-off function on \(\mathbb{R}^n\) such that
\[
\psi(x) :=
\begin{cases}
1, & \text{if } |x| < \dfrac{\rho}{2}, \\[4pt]
0, & \text{if } |x| > \rho.
\end{cases}
\]
For \(\lambda > 0\) sufficiently large, we define
\[
w_\lambda(x) = \psi(x)\, U_{(0,\lambda)}(x), \quad x \in \Omega,
\]
where \(U_{(0,\lambda)}\) is defined in \eqref{eq2.5}.
By construction, the test function \(w_\lambda \in V(\Omega)\setminus \{0\}\).

In order to expand \(J(w_\lambda)\), (see \eqref{eq2.7}), we need to estimate
\(\|w_\lambda\|_{V(\Omega)}^2\) and \(\|w_\lambda\|_{HLS}^2\).
For this purpose, we introduce the following Lemmas.

%%%%%%%%%%%%%%%%%%%%%%%%%%%%%%%%%%%%%%%%%%%%%%%%%%%%%%%%%%%%%
\begin{lemma}\label{lem3.1}
Assume that the expansion \eqref{1.1} holds. Then we have the following asymptotic estimates

\begin{enumerate}[label=(\roman*)]
    \item For \(\beta \in (1, n-1)\),
    \[
    \int_{\Omega} |\nabla w_\lambda|^2 \, dx = \frac{1}{2} S_{HL}^{\frac{2^*_{\mu}}{2^*_{\mu}-1}}
    - c_0 \tilde{C}_{n,\mu}^2 (n-2)^2 \bigl(n(n-2)\bigr)^{\frac{n-2}{2}}
    \frac{\sum_{k=1}^{n-1} \gamma_k}{\lambda^{\beta-1}}
    + O\!\left( \frac{1}{\lambda^{n-2}} \right),
    \]

    \item For \(\beta = n-1\),
    \[
    \int_{\Omega} |\nabla w_\lambda|^2 \, dx = \frac{1}{2} S_{HL}^{\frac{2^*_{\mu}}{2^*_{\mu}-1}}
    - \hat{c}\left( \sum_{k=1}^{n-1} \gamma_k \right) \frac{\log \lambda}{\lambda^{\beta-1}}
    + O\!\left( \frac{\log \lambda}{\lambda^{\beta-1}} \right).
    \]
\end{enumerate}

Here \(c_0 > 0\) is the constant defined by
\[
c_0 = \int_{\mathbb{R}^{n-1}} \frac{|z_1|^{\beta}}{(1+|z|^2)^{n-1}} \, dz
    - \int_{\mathbb{R}^{n-1}} \frac{|z_1|^{\beta}}{(1+|z|^2)^n} \, dz
    =: c_1 - c_2,
\]
\(\tilde{C}_{n,\mu}\) is defined in \eqref{eq2.5}, and \(\hat{c} > 0\) is a positive constant.
\end{lemma}
\begin{proof}
We begin with the identity
\begin{equation}\label{eq3.2}
\int_{\Omega} |\nabla w_\lambda|^2 dx = \int_{\Omega\cap B(0,\frac{\rho}{2})} |\nabla U_{(0,\lambda)}|^2 dx + O\!\left( \int_{|x| > \frac{\rho}{2}} |\nabla w_\lambda|^2 dx \right) =: K_1 + R_1.
\end{equation}
The remainder term \(R_1\) satisfies
\[
R_1 \leq C \left( \int_{|x| > \frac{\rho}{2}} |\nabla \delta_{(0,\lambda)}|^2 dx + \int_{|x| > \frac{\rho}{2}} |\nabla \delta_{(0,\lambda)}| \delta_{(0,\lambda)} dx + \int_{|x| > \frac{\rho}{2}} \delta_{(0,\lambda)}^2 dx \right),
\]
where \(\delta_{(0,\lambda)}\) is defined in \eqref{eq2.4}. Using the formula
\begin{equation}\label{eq2.9}
\nabla \delta_{(0,\lambda)}(x) = -(n-2)(n(n-2))^{\frac{n-2}{4}} \frac{\lambda^{\frac{n+2}{2}} x}{(1 + \lambda^2 |x|^2)^{\frac{n}{2}}},\quad x\in \mathbb{R}^n,
\end{equation}
we obtain
\begin{equation}\label{eq3.4}
R_1 = O\!\left( \frac{1}{(\lambda \rho)^{n-2}} \right) = O\!\left( \frac{1}{\lambda^{n-2}} \right),
\end{equation}
since \(\rho\) is a fixed constant independent of \(\lambda\).

We now estimate the main term \(K_1\) in \eqref{eq3.2}. Under condition \(\mathbf{(H_1)}\), we have \(\varphi(x') \geq 0\) for any \(x'\in\mathbb{R}^{n-1}\) and close to zero, so that
\[
B(0,\rho) \cap \Omega = \{(x', x_n) \in B(0,\rho) : x_n > \varphi(x')\} = B(0,\rho)^+ \cap \Omega,
\]
where \[
B(0,\rho)^+ = \{(x', x_n) \in B(0,\rho) : x_n > 0\}.
\]
 Define
\[
\Sigma := \{(x', x_n) \in B(0,\rho) : 0 < x_n \leq \varphi(x')\}.
\]
Then
\begin{equation}\label{eq3.5}
K_1 = \int_{B(0,\rho)^+} |\nabla U_{(0,\lambda)}|^2 dx - \int_{\Sigma} |\nabla U_{(0,\lambda)}|^2 dx.
\end{equation}
Notice that
\begin{equation}\label{eq3.6}
\begin{aligned}
\int_{B(0,\rho)^+} |\nabla U_{(0,\lambda)}|^2 \, dx
&= \frac{1}{2} \int_{B(0,\rho)} |\nabla U_{(0,\lambda)}|^2 \, dx \\
&= \frac{1}{2} \int_{\mathbb{R}^n} |\nabla U_{(0,\lambda)}|^2 \, dx
  + O\!\left( \frac{1}{\lambda^{n-2}} \right) \\
&= \frac{1}{2} S_{HL}^{\frac{2^*_{\mu}}{2^*_{\mu}-1}}
  + O\!\left( \frac{1}{\lambda^{n-2}} \right),
\end{aligned}
\end{equation}
since \(U_{(0,\lambda)}\) is a minimizer of \eqref{eq2.3} on \(\mathbb{R}^n\) and satisfies \eqref{eq2.6}.

Next we estimate the second integral in \eqref{eq3.5}. For a small \(\delta > 0\), set
\[
L_\delta = \{ x' = (x_1, \dots, x_{n-1}) \in \mathbb{R}^{n-1} : |x'| < \delta \},
\quad
\tilde{L}_\delta = \{ (x', x_n) \in \mathbb{R}^n : x' \in L_\delta \}.
\]
We split
\[
\int_{\Sigma} |\nabla U_{(0,\lambda)}|^2 dx = \int_{\Sigma \cap \tilde{L}_\delta} |\nabla U_{(0,\lambda)}|^2 dx + \int_{\Sigma \cap \tilde{L}_\delta^c} |\nabla U_{(0,\lambda)}|^2 dx
 \] 
 \[
 = \int_{\Sigma \cap \tilde{L}_\delta} |\nabla U_{(0,\lambda)}|^2 dx + O\!\left( \frac{1}{\lambda^{n-2}} \right).
\]
From \eqref{eq2.5} and \eqref{eq2.9} we obtain
\begin{equation}\label{eq2.13}
\begin{aligned}
\int_{\Sigma} |\nabla U_{(0,\lambda)}|^2 dx
&= \tilde{C}_{n,\mu}^2 (n-2)^2 (n(n-2))^{\frac{n-2}{2}} \int_{\Sigma \cap \tilde{L}_\delta} \frac{\lambda^{n+2} |x|^2}{(1+\lambda^2 |x|^2)^n} dx + O\!\left( \frac{1}{\lambda^{n-2}} \right) \\
&= \tilde{C}_{n,\mu}^2 (n-2)^2 (n(n-2))^{\frac{n-2}{2}} \Bigg[ \int_{\Sigma \cap \tilde{L}_\delta} \frac{\lambda^n}{(1+\lambda^2 |x|^2)^{n-1}} dx \\
&\quad - \int_{\Sigma \cap \tilde{L}_\delta} \frac{\lambda^n}{(1+\lambda^2 |x|^2)^n} dx \Bigg] + O\!\left( \frac{1}{\lambda^{n-2}} \right).
\end{aligned}
\end{equation}
The first integral in \eqref{eq2.13} can be rewritten as
\begin{equation*}
\begin{aligned}
\int_{\Sigma \cap \tilde{L}_\delta}
\frac{\lambda^n}{(1+\lambda^2 |x|^2)^{n-1}} \, dx
&= \int_{x' \in L_\delta} \int_0^{\varphi(x')}
\frac{\lambda^n}
     {(1+\lambda^2 |x'|^2 + \lambda^2 x_n^2)^{\,n-1}}
\, dx_n \, dx' \\
&= \int_{x' \in L_\delta} \int_0^{\varphi(x')}
\frac{\lambda^n}
     {(1+ \lambda^2 |x'|^2)^{\,n-1}
     \left[
     1 + \left(\frac{\lambda x_n}{\sqrt{1+\lambda^2 |x'|^2}}\right)^2
     \right]^{\,n-1}}
\, dx_n \, dx'.
\end{aligned}
\end{equation*}
With the substitution \(y = \frac{\lambda x_n}{\sqrt{1+\lambda^2 |x'|^2}}\), this becomes
\[
\int_{\Sigma \cap \tilde{L}_\delta} \frac{\lambda^n}{(1+\lambda^2 |x|^2)^{n-1}} dx = \int_{x' \in L_\delta} \int_0^{\frac{\lambda \varphi(x')}{\sqrt{1+\lambda^2 |x'|^2}}} \frac{\lambda^{n-1}}{(1+\lambda^2 |x'|^2)^{n-\frac{3}{2}}} \frac{dy}{(1+|y|^2)^{n-1}} dx'.
\]
For any \(\tilde{r} > 0\), a Taylor expansion around $0$ gives
\[
\int_0^T \frac{ds}{(1+s^2)^{\tilde{r}}} = T + O(T^3).
\]
Consequently,
\[
\int_{\Sigma \cap \tilde{L}_\delta} \frac{\lambda^n}{(1+\lambda^2 |x|^2)^{n-1}} dx = \int_{x' \in L_\delta} \frac{\lambda^n \varphi(x')}{(1+\lambda^2 |x'|^2)^{{n-1}}} dx' + O\!\left( \int_{x' \in L_\delta} \frac{\lambda^{n+2} \varphi(x')^3}{(1+\lambda^2 |x'|^2)^{{n}}} dx' \right).
\]
Using the expansion \eqref{1.1}, we obtain
\begin{equation*}
\begin{aligned}
\int_{\Sigma \cap \tilde{L}_\delta}
\frac{\lambda^n}{(1+\lambda^2 |x|^2)^{n-1}} \, dx
&= \sum_{k=1}^{n-1} \gamma_k
\int_{x' \in L_\delta}
\frac{\lambda^n |x'_k|^\beta}
     {(1+\lambda^2 |x'|^2)^{n-1}} \, dx'
 + o\!\left(
\int_{x' \in L_\delta}
\frac{\lambda^n |x'|^{\beta}}
     {(1+\lambda^2 |x'|^2)^{n-1}} \, dx'
\right) \\
&\quad + O\!\left(
\int_{x' \in L_\delta}
\frac{\lambda^{n+2} |x'|^{3\beta}}
     {(1+\lambda^2 |x'|^2)^{n}} \, dx'
\right).
\end{aligned}
\end{equation*}
Observe that for small \(\delta\),
\[
\frac{\lambda^2 |x'|^{3\beta}}{(1+\lambda^2 |x'|^2)^{{n}}} = o\left( \frac{|x'|^\beta}{(1+\lambda^2 |x'|^2)^{{n-1}}} \right),
\]
since
\[
\frac{\lambda^2 |x'|^{3\beta}}{(1+\lambda^2 |x'|^2)^{{n}}} \cdot \frac{(1+\lambda^2 |x'|^2)^{{n-1}}}{|x'|^\beta} = \frac{\lambda^2 |x'|^{2}}{1+\lambda^2 |x'|^{2}}|x'|^{2(\beta -1)} = O(\delta^{2(\beta-1)}) \to 0 \quad \text{as } \delta \to 0.
\]
Setting \(z=\lambda x'\), we find for \(\beta < n-1\),
\begin{equation}\label{eq2.14}
\int_{\Sigma \cap \tilde{L}_\delta} \frac{\lambda^n}{(1+\lambda^2 |x|^2)^{n-1}} dx = \frac{\sum_{k=1}^{n-1} \gamma_k}{\lambda^{\beta-1}} \int_{\mathbb{R}^{n-1}} \frac{|z_1|^\beta}{(1+|z|^2)^{n-1}} dz + o\left( \frac{1}{\lambda^{\beta-1}} \right),
\end{equation}
and for \(\beta = n-1\)
\begin{equation}\label{eq2.144}
\int_{\Sigma \cap \tilde{L}_\delta} \frac{\lambda^n}{(1+\lambda^2 |x|^2)^{n-1}} dx = \hat{c} \left( \sum_{k=1}^{n-1} \gamma_k \right) \frac{\log \lambda}{\lambda^{\beta-1}} + o\left( \frac{\log \lambda}{\lambda^{\beta-1}} \right).
\end{equation}
A similar computation yields for the second integral in \eqref{eq2.13}
\begin{equation}\label{eq2.15}
\int_{\Sigma \cap \tilde{L}_\delta} \frac{\lambda^n}{(1+\lambda^2 |x|^2)^n} dx = \frac{\sum_{k=1}^{n-1} \gamma_k}{\lambda^{\beta-1}} \int_{\mathbb{R}^{n-1}} \frac{|z_1|^\beta}{(1+|z|^2)^n} dz + o\left( \frac{1}{\lambda^{\beta-1}} \right),
\end{equation}
for any \(\beta \leq n-1\).

Combining \eqref{eq2.13}, \eqref{eq2.14} and \eqref{eq2.15}, we obtain for \(\beta < n-1\),
\begin{equation}\label{eq2.16}
\begin{aligned}
\int_{\Sigma} |\nabla U_{(0,\lambda)}|^2 dx
&= \tilde{C}_{n,\mu}^2 (n-2)^2 (n(n-2))^{\frac{n-2}{2}} \frac{\sum_{k=1}^{n-1} \gamma_k}{\lambda^{\beta-1}} \\
&\quad\times \left( \int_{\mathbb{R}^{n-1}} \frac{|z_1|^\beta}{(1+|z|^2)^{n-1}} dz - \int_{\mathbb{R}^{n-1}} \frac{|z_1|^\beta}{(1+|z|^2)^n} dz \right)
+ O\!\left( \frac{1}{\lambda^{n-2}} \right).
\end{aligned}
\end{equation}
For \(\beta = n-1\), we have
\begin{equation}\label{eq2.166}
\int_{\Sigma} |\nabla U_{(0,\lambda)}|^2 dx = \hat{c}  \left( \sum_{k=1}^{n-1} \gamma_k \right) \frac{\log \lambda}{\lambda^{\beta-1}} + o\!\left( \frac{\log \lambda}{\lambda^{\beta-1}} \right).
\end{equation}
Inserting these estimates together with \eqref{eq3.6} into \eqref{eq3.5} gives for \(\beta < n-1\),
\begin{equation}\label{eq2.17}
\begin{aligned}
K_1 &= \frac{1}{2} S_{HL}^{\frac{2^*_{\mu}}{2^*_{\mu}-1}}
- \tilde{C}_{n,\mu}^2 (n-2)^2 (n(n-2))^{\frac{n-2}{2}} \frac{\sum_{k=1}^{n-1} \gamma_k}{\lambda^{\beta-1}} \\
&\quad\times \left( \int_{\mathbb{R}^{n-1}} \frac{|z_1|^\beta}{(1+|z|^2)^{n-1}} dz - \int_{\mathbb{R}^{n-1}} \frac{|z_1|^\beta}{(1+|z|^2)^n} dz \right)
+ O\!\left( \frac{1}{\lambda^{n-2}} \right),
\end{aligned}
\end{equation}
and for \(\beta = n-1\)
\begin{equation}\label{eq2.177}
K_1 = \frac{1}{2} S_{HL}^{\frac{2^*_{\mu}}{2^*_{\mu}-1}} - \hat{c}  \left( \sum_{k=1}^{n-1} \gamma_k \right) \frac{\log \lambda}{\lambda^{\beta-1}} + o\left( \frac{\log \lambda}{\lambda^{\beta-1}} \right).
\end{equation}
Finally, putting \eqref{eq3.4}, \eqref{eq2.17} and \eqref{eq2.177} into \eqref{eq3.2} completes the proof of the lemma.
\end{proof}
\begin{lemma}\label{lem3.2}
The following two estimates hold:
\[
\begin{aligned}
&(E_1)\quad
\int_{\Omega} V(x)\, w_\lambda^2(x) \, dx
= \int_{\Omega \cap B\!\left(0,\frac{\rho}{2}\right)}
V(x)\, U_{(0,\lambda)}^2(x) \, dx
+ O\!\left( \frac{1}{\lambda^{n-2}} \right), \\[6pt]
&(E_2)\quad
\int_{\Omega} V(x)\, w_\lambda^2(x) \, dx =
\begin{cases}
O\!\left( \dfrac{1}{\lambda^{2}} \right), & n \geq 5, \\[6pt]
O\!\left( \dfrac{\log \lambda}{\lambda^{2}} \right), & n = 4, \\[6pt]
O\!\left( \dfrac{1}{\lambda} \right), & n = 3 .
\end{cases}
\end{aligned}
\]
\end{lemma}
\begin{proof}
We start by writing
\begin{equation}\label{eq3.15}
\int_{\Omega} V(x)\, w_\lambda^2(x) \, dx
= \int_{\Omega \cap B\!\left(0,\frac{\rho}{2}\right)}
V(x)\, U_{(0,\lambda)}^2(x) \, dx
+ O\!\left(
\int_{\frac{\rho}{2}<|x|<\rho}
\delta_{(0,\lambda)}^2(x) \, dx
\right).
\end{equation}
We observe that
\begin{equation}\label{eq3.16}
\begin{aligned}
\int_{\frac{\rho}{2}<|x|<\rho}
\delta_{(0,\lambda)}^2(x) \, dx
&= (n(n-2))^{\frac{n-2}{2}}
\int_{\frac{\rho}{2}<|x|<\rho}
\frac{\lambda^{n-2}}{(1+\lambda^2 |x|^2)^{n-2}} \, dx \\
&= (n(n-2))^2 \frac{1}{\lambda^2}
\int_{\frac{\lambda\rho}{2}<|z|<\lambda\rho}
\frac{dz}{(1+|z|^2)^{n-2}} \\
&= O\!\left( \frac{1}{\lambda^2} \right)
\int_{\frac{\lambda\rho}{2}}^{\lambda\rho}
\frac{dr}{r^{\,n-3}} \\
&= O\!\left( \frac{1}{\lambda^{n-2}} \right).
\end{aligned}
\end{equation}
The estimate \((E_1)\) follows directly from
\eqref{eq3.15} and \eqref{eq3.16}.

\medskip
We now prove the estimate \((E_2)\).
Since \(V(x)\) is bounded in \(\Omega\), we obtain
\begin{equation}\label{eq3.17}
\begin{aligned}
\int_{\Omega \cap B\!\left(0,\frac{\rho}{2}\right)}
V(x)\, U_{(0,\lambda)}^2(x) \, dx
&= O\!\left(
\int_{|x|<\frac{\rho}{2}}
\frac{\lambda^{n-2}}{(1+\lambda^2 |x|^2)^{n-2}} \, dx
\right) \\
&= O\!\left(
\frac{1}{\lambda^2}
\int_{|z|<\frac{\lambda\rho}{2}}
\frac{dz}{(1+|z|^2)^{n-2}}
\right).
\end{aligned}
\end{equation}
Let \(A > 0\) be a sufficiently large fixed constant. Then
\begin{equation}\label{eq3.18}
\begin{aligned}
\int_{|z|<\frac{\lambda\rho}{2}}
\frac{dz}{(1+|z|^2)^{n-2}}
&= O\!\left(
\int_{0}^{A}
\frac{r^{n-1}}{(1+r^2)^{n-2}} \, dr
\right)
+ O\!\left(
\int_{A}^{\lambda\rho}
\frac{dr}{r^{\,n-3}}
\right) \\
&= O(1) +
\begin{cases}
O\!\left( \dfrac{1}{\lambda^{n-4}} \right), & n \geq 5, \\[6pt]
O(\log \lambda), & n = 4, \\[6pt]
O(\lambda), & n = 3 .
\end{cases}
\end{aligned}
\end{equation}
Combining \eqref{eq3.15}--\eqref{eq3.18}, the estimate \((E_2)\) follows.
This completes the proof of this lemma.
\end{proof}

%%%%%%%%%%%%%%%%%%%%%%%%%%%%%%%%%%%%%%%%%%%%%%%%%%%%%%%%%%%%%%%%%%%%%%%%%%%%%%%%%%%%%%%%%%%
\begin{lemma}\label{lem3.3}
Under the assumptions of Theorems \ref{th1.3} and \ref{th1.4}, the following estimates hold:
\begin{equation}
\begin{aligned}
\|w_\lambda\|_{V(\Omega)}^2 \leq\;
&\frac{1}{2} S_{HL}^{\frac{2^*_{\mu}}{2^*_{\mu}-1}}
\Bigg(
1 -
\frac{
2(c_1 - c_2)\tilde{C}_{n,\mu}^2 (n-2)^2 (n(n-2))^{\frac{n-2}{2}}
\sum_{k=1}^{n-1} \gamma_k
}{
S_{HL}^{\frac{2^*_{\mu}}{2^*_{\mu}-1}} \lambda^{\beta-1}
}
\Bigg)  \\
&\quad +\; o\!\left( \frac{1}{\lambda^{\beta-1}} \right),
\qquad \beta < n-1,
\end{aligned}
\end{equation}
and
\begin{equation}
\|w_\lambda\|_{V(\Omega)}^2 \leq
\frac{1}{2} S_{HL}^{\frac{2^*_{\mu}}{2^*_{\mu}-1}}
\left(
1 - \frac{2 \hat{c} \sum_{k=1}^{n-1} \gamma_k}
{S_{HL}^{\frac{2^*_{\mu}}{2^*_{\mu}-1}}}
\frac{\log \lambda}{\lambda^{\beta-1}}
\right)
+ o\left( \frac{\log \lambda}{\lambda^{\beta-1}} \right),
\quad \beta = n-1.
\end{equation}
The constants \(c_1\), \(c_2\), and \(\hat{c}\) are defined in Lemma \ref{lem3.1}.
\end{lemma}

\begin{proof}
In \(V(\Omega)\), we have
\[
\|w_\lambda\|_{V(\Omega)}^2
= \int_{\Omega} |\nabla w_\lambda|^2 \, dx
+ \int_{\Omega} V(x) w_\lambda^2(x) \, dx.
\]
Using the first expansion given in Lemma \ref{lem3.1}, we have
\begin{equation}\label{eq3.21}
\begin{aligned}
\int_{\Omega} |\nabla w_\lambda|^2 \, dx
=\;& \frac{1}{2} S_{HL}^{\frac{2^*_{\mu}}{2^*_{\mu}-1}}
\Bigg(
1 -
\frac{
2(c_1 - c_2)\tilde{C}_{n,\mu}^2 (n-2)^2 (n(n-2))^{\frac{n-2}{2}}
\sum_{k=1}^{n-1} \gamma_k
}{
S_{HL}^{\frac{2^*_{\mu}}{2^*_{\mu}-1}} \lambda^{\beta-1}
}
\Bigg)  \\
&\quad +\; o\!\left( \frac{1}{\lambda^{\beta-1}} \right),
\qquad \beta < n-1 .
\end{aligned}
\end{equation}
Moreover, by the second expansion of Lemma \ref{lem3.1}, we have
\begin{equation}\label{eq3.22}
\int_{\Omega} |\nabla w_\lambda|^2 \, dx
= \frac{1}{2} S_{HL}^{\frac{2^*_{\mu}}{2^*_{\mu}-1}}
\left(
1 - \frac{2 \hat{c} \sum_{k=1}^{n-1} \gamma_k}
{S_{HL}^{\frac{2^*_{\mu}}{2^*_{\mu}-1}}}
\frac{\log \lambda}{\lambda^{\beta-1}}
\right)
+ o\left( \frac{\log \lambda}{\lambda^{\beta-1}} \right),
\quad \beta = n-1.
\end{equation}
Assume first that the hypotheses of Theorem \ref{th1.3} are satisfied, namely the
\(\mathbf{(H_1)}\) condition with \(\beta < 3\). Then, for \(n \geq 4\), we have
\begin{equation}\label{eq3.23}
\int_{\Omega} V(x) w_\lambda^2(x) \, dx
= o\!\left( \frac{1}{\lambda^{\beta-1}} \right).
\end{equation}
Indeed, by estimate \((E_2)\) of Lemma \ref{lem3.2}, we have
\(\lambda^{-2} = o(\lambda^{-(\beta-1)})\) for \(n \geq 5\), and
\(\log \lambda / \lambda = o(\lambda^{-(\beta-1)})\) for \(n = 4\).

If \(n = 3\), again by \((E_2)\) of Lemma \ref{lem3.2}, we have
\begin{equation}\label{eq3.24}
\int_{\Omega} V(x) w_\lambda^2(x) \, dx =
\begin{cases}
o\left( \dfrac{1}{\lambda^{\beta-1}} \right), & \beta < n-1, \\[6pt]
o\left( \dfrac{\log \lambda}{\lambda^{\beta-1}} \right), & \beta = n-1.
\end{cases}
\end{equation}
Combining \eqref{eq3.21}--\eqref{eq3.24}, the desired estimates of the Lemma follow in this case.

Now assume that the hypotheses of Theorem \ref{th1.4} hold. Using estimate \((E_1)\)
of Lemma \ref{lem3.2}, we have
\begin{equation}\label{eq3.25}
\int_{\Omega} V(x) w_\lambda^2(x) \, dx
\leq O\!\left( \frac{1}{\lambda^{n-2}} \right)
=
\begin{cases}
o\left( \dfrac{1}{\lambda^{\beta-1}} \right), & \beta < n-1, \\[6pt]
o\left( \dfrac{\log \lambda}{\lambda^{\beta-1}} \right), & \beta = n-1.
\end{cases}
\end{equation}
The conclusion of this Lemma  follows from
\eqref{eq3.21}, \eqref{eq3.22}, and \eqref{eq3.25}.
\end{proof}

\begin{lemma}\label{lem3.4}
Under the assumptions of Theorems \ref{th1.3} and \ref{th1.4}, we have
\[
\|w_\lambda\|_{HLS}^2
= \left( \frac{1}{2} \right)^{\frac{2}{2^*_\mu}}
S_{HL}^{\frac{1}{2^*_\mu -1}}
\left(
1 - \frac{4 c_2 \tilde{C}_{n,\mu}^2 (n(n-2))^{\frac{n}{2}}
\sum_{k=1}^{n-1} \gamma_k}
{2^*_\mu S_{HL}^{\frac{2^*_\mu}{2^*_\mu-1}} \lambda^{\beta-1}}
\right)
+ o\left( \frac{1}{\lambda^{\beta-1}} \right),
\]
where \(\tilde{C}_{n,\mu}\) is defined in \eqref{eq2.5} and \(c_2\) is given in Lemma \ref{lem3.1}.
\end{lemma}

\begin{proof}
 We have
\[
\|w_\lambda\|_{HLS}^2
= \left(
\int_{\Omega} (I_\mu * w_\lambda^{2^*_\mu})(x)\,
w_\lambda^{2^*_\mu}(x)\, dx
\right)^{\frac{1}{2^*_\mu}}
=: D^{\frac{1}{2^*_\mu}}.
\]
Let
\[
\overline{\alpha}_{n,\mu}
= \frac{\Gamma\!\left(\frac{\mu}{2}\right)}
{\Gamma\!\left(\frac{n-\mu}{2}\right)
\pi^{\frac{n}{2}} 2^{\,n-\mu}}.
\]
Following the computations of [\cite{giacomoni2023critical}, Section~3], we write
\begin{equation}\label{eq3.26}
\begin{aligned}
D
&= \int_{\Omega} \int_{\Omega}
\frac{\overline{\alpha}_{n,\mu}}{|x-y|^\mu}
w_\lambda^{2^*_\mu}(y) w_\lambda^{2^*_\mu}(x)\, dx\, dy \\
&= \frac{1}{4}
\int_{B(0,\frac{\rho}{2})} \int_{B(0,\frac{\rho}{2})}
\frac{\overline{\alpha}_{n,\mu}}{|x-y|^\mu}
U_{(0,\lambda)}^{2^*_\mu}(y) U_{(0,\lambda)}^{2^*_\mu}(x)\, dx\, dy \\
&\quad
- \int_{\Sigma} \int_{B(0,\frac{\rho}{2})}
\frac{\overline{\alpha}_{n,\mu}}{|x-y|^\mu}
U_{(0,\lambda)}^{2^*_\mu}(y) U_{(0,\lambda)}^{2^*_\mu}(x)\, dx\, dy
+ R_1 + R_2 + R_3 ,
\end{aligned}
\end{equation}
where
\begin{align*}
R_1 &=
O\!\left(
\int_{|x|>\frac{\rho}{2}} \int_{\Omega}
\frac{1}{|x-y|^\mu}
U_{(0,\lambda)}^{2^*_\mu}(y) U_{(0,\lambda)}^{2^*_\mu}(x)\, dx\, dy
\right)
= O\!\left( \frac{1}{\lambda^\frac{2n-\mu}{2}} \right)
= O\!\left( \frac{1}{\lambda^{n-2}} \right),
\quad \mu \le 4, \\[4pt]
R_2 &=
O\!\left(
\int_{|x|>\frac{\rho}{2}} \int_{|y|>\frac{\rho}{2}}
\frac{1}{|x-y|^\mu}
U_{(0,\lambda)}^{2^*_\mu}(y) U_{(0,\lambda)}^{2^*_\mu}(x)\, dx\, dy
\right)
= O\!\left( \frac{1}{\lambda^{2n-\mu}} \right)
= O\!\left( \frac{1}{\lambda^{n-2}} \right), \\[4pt]
R_3 &=
O\!\left(
\int_{\Sigma} \int_{\Sigma}
\frac{1}{|x-y|^\mu}
U_{(0,\lambda)}^{2^*_\mu}(y) U_{(0,\lambda)}^{2^*_\mu}(x)\, dx\, dy
\right)= O\left( \frac{1}{\lambda^{2n-\mu}} \right)+ O\left( \frac{1}{\lambda^\frac{(2n-\mu)(\beta-1)}{n}} \right) \\&\quad =
o\left( \frac{1}{\lambda^{\beta-1}} \right).
\end{align*}
We now compute the first integral of \eqref{eq3.26}. We have
\begin{equation}\label{eq3.27}
\begin{aligned}
&\int_{B(0,\frac{\rho}{2})} \int_{B(0,\frac{\rho}{2})}
\frac{\overline{\alpha}_{n,\mu}}{|x-y|^\mu}
U_{(0,\lambda)}^{2^*_\mu}(y) U_{(0,\lambda)}^{2^*_\mu}(x)\, dx\, dy \\
&= \int_{\mathbb{R}^n} \int_{\mathbb{R}^n}
\frac{\overline{\alpha}_{n,\mu}}{|x-y|^\mu}
U_{(0,\lambda)}^{2^*_\mu}(y) U_{(0,\lambda)}^{2^*_\mu}(x)\, dx\, dy
+ R_1+ R_2 \\
&= \int_{\mathbb{R}^n}
(I_\mu * U_{(0,\lambda)}^{2^*_\mu})(x)
U_{(0,\lambda)}^{2^*_\mu}(x)\, dx
+ O\!\left( \frac{1}{\lambda^{n-2}} \right) \\
&= S_{HL}^{\frac{2^*_\mu}{2^*_\mu-1}}
+ O\!\left( \frac{1}{\lambda^{n-2}} \right),
\end{aligned}
\end{equation}
since \(U_{(0,\lambda)}\) is a minimizer of \eqref{eq2.3} in \(\mathbb{R}^n\) and satisfies \eqref{eq2.6}.

We now estimate the second integral of \eqref{eq3.26}. Note that from \eqref{eq2.4}, the function \(\delta_{(0,\lambda)}\) satisfies \[ -\Delta \delta_{(0,\lambda)} = \delta_{(0,\lambda)}^{2^*-1} \quad \text{in } \mathbb{R}^n, \] where \(2^* = \frac{2n}{n-2}\).
Therefore, from \eqref{eq2.4} and \eqref{eq2.5}, we obtain
\begin{equation}\label{eq3.28}
I_\mu * U_{(0,\lambda)}^{2^*_\mu}
= \frac{1}{\tilde{C}_{n,\mu}^{\frac{4}{n-2}}}
U_{(0,\lambda)}^{2^*-2^*_\mu}
\quad \text{in } \mathbb{R}^n.
\end{equation}
It follows that
\begin{align*}
&\int_{\Sigma} \int_{B(0,{\rho})}
\frac{\overline{\alpha}_{n,\mu}}{|x-y|^\mu}
U_{(0,\lambda)}^{2^*_\mu}(y) U_{(0,\lambda)}^{2^*_\mu}(x)\, dx\, dy \\
&= \int_{\Sigma} \int_{\mathbb{R}^n}
\frac{\overline{\alpha}_{n,\mu}}{|x-y|^\mu}
U_{(0,\lambda)}^{2^*_\mu}(y) U_{(0,\lambda)}^{2^*_\mu}(x)\, dx\, dy
+ R_1 \\
&= \frac{1}{\tilde{C}_{n,\mu}^{\frac{4}{n-2}}}
\int_{\Sigma} U_{(0,\lambda)}^{2^*}(x) \, dx
+ O\!\left( \frac{1}{\lambda^{n-2}} \right),
\end{align*}
since \(U_{(0,\lambda)}\) satisfies \eqref{eq3.28}.
Using now \eqref{eq2.4} and \eqref{eq2.5}, we obtain
\[
\int_{\Sigma} \int_{B(0,\frac{\rho}{2})}
\frac{\overline{\alpha}_{n,\mu}}{|x-y|^\mu}
U_{(0,\lambda)}^{2^*_\mu}(y) U_{(0,\lambda)}^{2^*_\mu}(x)\, dx\, dy
= \tilde{C}_{n,\mu}^{2} (n(n-2))^{\frac{n}{2}}
\int_{\Sigma} \frac{\lambda^n}{(1+\lambda^2 |x|^2)^n}\, dx .
\]
Recall from \eqref{eq2.15} that
\[
\int_{\Sigma \cap \tilde{L}_{\delta}}
\frac{\lambda^n}{(1+\lambda^2 |x|^2)^n}\, dx
= c_2 \frac{\sum_{k=1}^{n-1} \gamma_k}{\lambda^{\beta-1}}
+ o\left( \frac{1}{\lambda^{\beta-1}} \right),
\]
where
\[
c_2 = \int_{\mathbb{R}^{n-1}} \frac{|z_1|^\beta}{(1+|z|^2)^n}\, dz.
\]
Moreover, by a direct computation,
\[
\int_{\Sigma \cap \tilde{L}_{\delta}^{\,c}}
\frac{\lambda^n}{(1+\lambda^2 |x|^2)^n}\, dx
= O\!\left( \int_{|z|>\lambda \delta}
\frac{dz}{(1+|z|^2)^n} \right)
= O\!\left( \frac{1}{\lambda^n} \right).
\]
Therefore, we conclude that
\begin{equation}\label{eq3.281}
\begin{aligned}
\int_{\Sigma} \int_{B(0,\frac{\rho}{2})}
\frac{\overline{\alpha}_{n,\mu}}{|x-y|^\mu}
U_{(0,\lambda)}^{2^*_\mu}(y)\,
U_{(0,\lambda)}^{2^*_\mu}(x)\, dx\, dy
=\;&
c_2 \tilde{C}_{n,\mu}^{2} (n(n-2))^{\frac{n}{2}}
\frac{\sum_{k=1}^{n-1} \gamma_k}{\lambda^{\beta-1}} \\
&\quad +\; o\!\left( \frac{1}{\lambda^{\beta-1}} \right),
\qquad \text{for any } \beta \le n-1 .
\end{aligned}
\end{equation}
Combining estimates \eqref{eq3.26}--\eqref{eq3.27} with \eqref{eq3.281}, we obtain
\[
D
= \frac{1}{4} S_{HL}^{\frac{2^*_\mu}{2^*_\mu-1}}
\left(
1 - \frac{4 c_2 \tilde{C}_{n,\mu}^{2} (n(n-2))^{\frac{n}{2}}
\sum_{k=1}^{n-1} \gamma_k}
{S_{HL}^{\frac{2^*_\mu}{2^*_\mu-1}} \lambda^{\beta-1}}
\right)
+ o\left( \frac{1}{\lambda^{\beta-1}} \right).
\]
Hence, the proof follows from the above expansions.
\end{proof}
%%%%%%%%%%%%%%%%%%%%%%%%%%%%%%%%%%%%%%%%%%
\begin{lemma}\label{lem3.5}
Under the assumptions of Theorems \ref{th1.3} and \ref{th1.4}, inequality \eqref{eq3.1} holds.
\end{lemma}

\begin{proof}
Let \(\lambda>0\) be sufficiently large. By \eqref{eq2.7}, we have
\[
J(w_\lambda)=\frac{\|w_\lambda\|_{V(\Omega)}^2}{\|w_\lambda\|_{HLS}^2}.
\]
Using the expansions of Lemmas \ref{lem3.3} and \ref{lem3.4}, we obtain for \(\beta<n-1\),
\begin{equation}\label{eq3.31}
\begin{aligned}
J(w_\lambda) \leq {} &
\left( \frac{1}{2} \right)^{\frac{2^*_\mu-2}{2^*_\mu}} S_{HL}
\Bigg[
1 - \frac{2}{S_{HL}^{\frac{2^*_\mu}{2^*_\mu-1}}}
\tilde{C}_{n,\mu}^2 (n(n-2))^{\frac{n-2}{2}}(n-2)  \\
&\qquad \times
\left( (n-2)c_1 -
\left( n-2+\frac{2}{2^*_\mu}n \right)c_2 \right)
\frac{\sum_{k=1}^{n-1} \gamma_k}{\lambda^{\beta-1}}
\Bigg]
+ o\left( \frac{1}{\lambda^{\beta-1}} \right),
\end{aligned}
\end{equation}
where the constants \(c_1\) and \(c_2\) are defined in Lemma \ref{lem3.1}.

We now claim that
\begin{equation}\label{eq3.32}
\bar{c}
:= (n-2)c_1 - \left( n-2+\frac{2}{2^*_\mu}n \right)c_2 > 0 .
\end{equation}
Indeed,
\[
\bar{c}
= (n-2)\left(
\int_{\mathbb{R}^{n-1}} \frac{|z_1|^\beta}{(1+|z|^2)^{n-1}}\, dz
- \left(1+\frac{2n}{2n-\mu}\right)
\int_{\mathbb{R}^{n-1}} \frac{|z_1|^\beta}{(1+|z|^2)^n}\, dz
\right).
\]
Since \(0<\mu\le 4\), we obtain
\[
\bar{c}
\ge (n-2)\left(
\int_{\mathbb{R}^{n-1}} \frac{|z_1|^\beta}{(1+|z|^2)^{n-1}}\, dz
- \frac{2(n-1)}{n-2}
\int_{\mathbb{R}^{n-1}} \frac{|z_1|^\beta}{(1+|z|^2)^n}\, dz
\right)
\]
\[
\ge (n-2)\int_{\mathbb{R}^{n-1}}
\frac{|z_1|^\beta\left(|z|^2-\frac{n}{n-2}\right)}{(1+|z|^2)^n}\, dz.
\]
Setting \(y=\frac{z}{\sqrt{\frac{n}{n-2}}}\), we infer that
\begin{align*}
\bar{c}
\ge {} & (n-2)
\left(\sqrt{\frac{n}{n-2}}\right)^{\beta+1}
\left(\frac{n-2}{n}\right)^{n-1}
\int_{\mathbb{R}^{n-1}}
\frac{|y_1|^\beta (|y|^2-1)}
{\left(\frac{n-2}{n}+|y|^2\right)^n}\, dy \\
\ge {} & (n-2)
\left(\sqrt{\frac{n}{n-2}}\right)^{\beta+1}
\left(\frac{n-2}{n}\right)^{n-1}
\int_{\mathbb{R}^{n-1}}
\frac{|y_1|^\beta (|y|^2-1)}
{(1+|y|^2)^n}\, dy .
\end{align*}
Using polar coordinates and a direct computation, we find
\[
\int_{\mathbb{R}^{n-1}}
\frac{|y_1|^\beta(|y|^2-1)}{(1+|y|^2)^n}\, dy
= c \int_1^{+\infty}
\frac{r^{n-2-\beta}(r^2-1)(r^{2\beta}-1)}{(1+r^2)^n}\, dr
>0,
\]
which proves the claim \eqref{eq3.32}.

Combining \eqref{eq3.31} and \eqref{eq3.32}, we write 
\[
J(w_\lambda)
\le \left( \frac{1}{2} \right)^{\frac{2^*_\mu-2}{2^*_\mu}} S_{HL}
\left(
1 - c \frac{\sum_{k=1}^{n-1} \gamma_k}{\lambda^{\beta-1}}(1+o(1))
\right),
\]
where \(c>0\). Hence, for \(\lambda\) sufficiently large,
\[
J(w_\lambda)
< \left( \frac{1}{2} \right)^{\frac{2^*_\mu-2}{2^*_\mu}} S_{HL},
\]
since by assumption \(\mathbf{(H_1)}\), \(\sum_{k=1}^{n-1} \gamma_k>0\).

If \(\beta=n-1\), Lemma \ref{lem3.4} yields
\begin{equation}\label{eq3.33}
\|w_\lambda\|_{HLS}^2
= \left( \frac{1}{2} \right)^{\frac{2}{2^*_\mu}}
S_{HL}^{\frac{1}{2^*_\mu-1}}
\left(1+ o\left(\frac{\log\lambda}{\lambda^{\beta-1}}\right)\right).
\end{equation}
Together with the second estimate in Lemma \ref{lem3.3} and \eqref{eq3.33}, we get
\[
J(w_\lambda)
< \left( \frac{1}{2} \right)^{\frac{2^*_\mu-2}{2^*_\mu}} S_{HL}.
\]
This completes the proof.
\end{proof}

\begin{proof}[\textbf{Proof of Theorems 1.3 and 1.4}]
The result follows directly from Lemma \ref{lem3.5} and [\cite{giacomoni2023critical}, Lemma~3.2].
\end{proof}

\section*{Declarations}

\noindent\textbf{Ethical Approval.}
Not applicable.

\medskip
\noindent\textbf{Competing interests.}
The authors declare that they have no competing interests.

\medskip
\noindent\textbf{Authors’ contributions.}
The authors contributed equally to this work.

\medskip
\noindent\textbf{Availability of data and materials.}
Data sharing is not applicable to this article as no new data were created or analyzed in this study.

\bibliographystyle{plain}
\bibliography{references}

\end{document}